\documentclass[notitlepage,12pt,reqno]{amsart}
\usepackage{amssymb,amsmath,amscd,amsthm,mathrsfs}
\usepackage[breaklinks=true]{hyperref}
\usepackage[usenames]{color}
\usepackage{scalerel}
\usepackage{enumerate,verbatim,soul,enumitem}

\newcommand{\R}{\ensuremath{\mathbb R}}
\newcommand{\C}{\ensuremath{\mathbb C}}
\newcommand{\N}{\ensuremath{\mathbb N}}
\newcommand{\Z}{\ensuremath{\mathbb Z}}

\newcommand{\D}{d}

\def\gcd{\mathop{\rm gcd}}
\def\TN{\rm TN}
\def\TP{\rm TP}

\newtheorem{theorem}[equation]{Theorem}
\newtheorem{proposition}[equation]{Proposition}

\newtheorem{corollary}[equation]{Corollary}
\newtheorem{lemma}[equation]{Lemma}

\theoremstyle{definition}
\newtheorem{remark}[equation]{Remark}
\newtheorem{definition}[equation]{Definition}

\numberwithin{equation}{section}

\begin{document}

\title{Total nonnegativity of GCD matrices and kernels}

\author[D.~Guillot]{Dominique Guillot}
\author[J.~Wu]{Jiaru Wu}

\date{\today}

\begin{abstract}
Let $X = (x_1,\dots,x_n)$ be a vector of distinct positive integers. The $n \times n$ matrix $S = S(X) := (\gcd(x_i,x_j))_{i,j=1}^n$, where $\gcd(x_i,x_j)$ denotes the greatest common divisor of $x_i$ and $x_j$, is called the greatest common divisor (GCD) matrix on $X$. By a surprising result of Beslin and Ligh [Linear Algebra and Appl. 118], all GCD matrices are positive definite. In this paper, we completely characterize the GCD matrices satisfying the stronger property of being totally nonnegative (TN) or totally positive (TP). As we show, a GCD matrix is never TP when $n \geq 3$, and is TN if and only if it is $TN_2$, i.e., all its $2 \times 2$ minors are nonnegative. We next demonstrate that a GCD matrix is $\TN_2$ if and only if the exponents of each prime divisor in the prime factorization of the $x_i$s form a monotonic sequence. Reformulated in the language of kernels, our results characterize the subsets of integers over which the kernel $K(x,y) = \gcd(x,y)$ is totally nonnegative. The proofs of our characterizations depend on Gantmacher and Krein's notion of a Green's matrix. We conclude by showing that a GCD matrix is TN if and only if it is a Green's matrix. As a consequence, we obtain explicit formulas for all the minors and for the inverse of totally nonnegative GCD matrices. 
\end{abstract}
\maketitle

\section{Introduction and Main Results}
We begin by setting some notation. We denote the set of positive and non-negative integers by $\N$ and $\Z_{\geq 0}$ respectively. Given two integers $a,b \in \N$, we write $a\ |\ b$ if $a$ divides $b$, i.e., if $b = \lambda a$ for some $\lambda \in \N$. We denote the greatest common divisor of two integers $m,n \in \N$ by $\gcd(m,n)$. More generally, we denote by $\gcd(n_1,n_2,\dots,n_k)$ the greatest common divisor of $n_1,\dots, n_k$.

\subsection{GCD matrices}
\begin{definition}
Let $X = (x_1,\dots,x_n) \in \N^n$ be a vector of positive integers. The $n \times n$ matrix $S = S(X) := (s_{ij})_{i,j=1}^n$ where $s_{ij} = \gcd(x_i, x_j)$ is called the {\it greatest common divisor (GCD) matrix} on $X$.
\end{definition}

\noindent The study of GCD matrices goes back to 1875, when H.~J.~S.~Smith  \cite{smith1875} was able to compute the determinant of $S((1,2,\dots,n))$: 
\[
\det (\gcd(i,j))_{i,j=1}^n = \phi(1) \phi(2) \dots \phi(n), 
\]
where $\phi$ denotes Euler's totient function. Almost 150 years later, these matrices and their generalizations continue to attract the attention of number theorists and linear algebraists (see e.g.~\cite{Alt2017,Haukkanen2018,Haukkanen1997, Ilmonen2018, Lin2018, Lindstrom1969, Mattila2014, bhat1991, Wilf1968} and the references therein). Of particular interest is a surprising result of Beslin and Ligh which shows that GCD matrices are always positive definite.
\begin{theorem}[{Beslin and Ligh \cite[Theorem 2]{beslin1989greatest}}]\label{Tbeslin}
Let $X = (x_1,\dots,x_n)$ be a vector of distinct positive integers. Then the GCD matrix $S(X)$ is positive definite.
\end{theorem}

\noindent Indeed, let $D = \{d_1,\dots,d_m\}$ be a set containing $X$ and all the divisors of the integers in $X$. Define a $n \times m$ matrix $A = (a_{ij})$ by 
\[
a_{ij} := e_{ij} \cdot \sqrt{\phi(d_j)}, 
\]
where 
\[
e_{ij} := \begin{cases}
1 & \textrm{if } d_j\ |\ x_i \\
0 & \textrm{otherwise},
\end{cases}
\]
and where $\phi(x)$ is Euler's totient function. By a direct calculation, 
\begin{equation*}
(AA^T)_{ij} = \sum_{k=1}^m a_{ik} a_{jk} = \sum_{\substack{d_k | x_i \\ d_k | x_j}} \sqrt{\phi(d_k)}\sqrt{\phi(d_k)} 
= \sum_{d_k | \gcd(x_i,x_j)} \phi(d_k) 
= \gcd(x_i,x_j).
\end{equation*} 
The last equality follows from the well-known identity $\sum_{d \mid x} \phi(d) = x$. Hence $S(X) = AA^T$ and so $S(X)$ positive semidefinite. A refined analysis of the rank of the previous matrices shows that $S(X)$ is also non-singular (see \cite{beslin1989greatest} for the details).

\begin{remark}\label{RBeslin}
Notice that if the $x_i$s are not distinct in Theorem \ref{Tbeslin}, the matrix $S(X)$ is still positive semidefinite, but is singular.
\end{remark}

The main goal of the paper is to determine which GCD matrices satisfy the stronger property of being {\it totally nonnegative}.

\subsection{Main results}
Recall that a real symmetric matrix is positive definite (resp.~semidefinite) if and only if its {\it principal} minors are positive (resp.~nonnegative). A stronger notion is that of totally positive (resp.~nonnegative) matrices, where {\it all} minors are required to be positive (resp.~nonnegative). These matrices arise in several areas including  approximation theory \cite{gasca2013total}, cluster algebras \cite{berenstein1996parametrizations,fomin1999double,fomin2000total}, combinatorics \cite{brenti1989unimodal,brenti1995combinatorics},  integrable systems \cite{kodama2011kp, kodama2014kp}, network analysis
\cite{postnikov2006total}, oscillatory matrices \cite{ando1987totally}, and representation theory \cite{lusztig123total,lusztig9total,rietsch2003totally}. 

More precisely, let $\alpha \subseteq \{1,\dots,m\}, \beta \subseteq \{1,\dots,n\}$, and let $A \in \R^{m \times n}$. We denote by $A[\alpha,\beta]$ the {\it submatrix} of $A$ with row indices in $\alpha$ and column indices in $\beta$, i.e.,
\[
A[\alpha,\beta] := (a_{ij})_{i \in \alpha, j \in \beta}.
\] 
When $\alpha = \beta$, we let $A[\alpha] := A[\alpha, \alpha]$ to simplify the notation.
\begin{definition}
Let $A = (a_{ij}) \in \R^{m \times n}$. The matrix $A$ is said to be:
\begin{itemize}
\item {\it totally nonnegative} if $\det A[\alpha, \beta] \geq 0$ for all $\alpha \subseteq \{1,\dots,m\}, \beta \subseteq \{1,\dots,n\}$ with $|\alpha| = |\beta|$. 
\item {\it totally positive} if $\det A[\alpha, \beta] > 0$  for all $\alpha \subseteq \{1,\dots,m\}, \beta \subseteq \{1,\dots,n\}$ with $|\alpha| = |\beta|$. 
\item $\TN_k$ if all minors of $A$ of size $\leq k$ are nonnegative.
\end{itemize}

\noindent We write $A \in \TN$, $A \in \TP$, or $A \in \TN_k$ to denote the fact that $A$ is TN, TP, or $\TN_k$, respectively.
\end{definition}

Following Beslin and Ligh's result (Theorem \ref{Tbeslin}), it is natural to examine which GCD matrices are totally nonnegative or totally positive. Our first main result shows that a GCD matrix is $\TN$ if and only if it is $\TN_2$.

\begin{theorem}[Main Result 1]\label{Tmain1}
Let $n \geq 3$ and let $X = (x_1,\dots,x_n) \in \N^n$. Then the following are equivalent for the GCD matrix $S(X)$: 
\begin{enumerate}
\item $S(X)$ is $\TN_2$.
\item $S(X)$ is $\TN$.
\item For $1 \leq i \leq j \leq k \leq n$, 
\begin{enumerate}
\item $\gcd(x_i,x_k) = \gcd(x_i,x_j,x_k)$ and 
\item $x_j \cdot \gcd(x_i, x_k) \mid x_i x_k$.
\end{enumerate}
\item For $1 \leq i \leq j \leq k \leq n$, $\gcd(x_i,x_j)\gcd(x_j,x_k) = x_j \cdot \gcd(x_i,x_k)$.
\end{enumerate}
Moreover, the above conditions imply that for all $1 \leq i \leq j \leq k \leq l \leq n$, 
\begin{equation}\label{Etn2}
\gcd(x_i, x_k) \cdot \gcd(x_j,x_l) = \gcd(x_i,x_l)\cdot \gcd(x_j,x_k).
\end{equation}
\end{theorem}

As a consequence, we immediately obtain that no GCD matrix is TP when $n \geq 3$.
\begin{corollary}
Let $n \geq 3$ and let $X = (x_1,\dots,x_n) \in \N^n$. Then $S(X)$ is not TP.
\end{corollary}
\begin{proof}
Observe that Equation \eqref{Etn2} can equivalently be formulated as 
\begin{equation}\label{Edet0}
\det A[\{i,j\}, \{k,l\}] = 0 \qquad \forall\ 1 \leq i < j \leq k < l \leq n.
\end{equation}
Hence a GCD matrix cannot be $\TP$ when $n \geq 3$.
\end{proof}

While Theorem \ref{Tmain1} reduces verifying the total nonnegativity of GCD matrices to computing $2 \times 2$ minors, it is not clear {\it a priori} which vectors $X = (x_1,\dots,x_n)$ yield TN matrices $S(X)$. Our second main result provides an explicit description of these vectors. To state the result, we use the following notation. Let $2 = p_1 < p_2 < p_3 < \dots$ denote the list of all prime numbers in increasing order. For a given integer $x \in \N$, we denote by $e_t(x) \in \Z_{\geq 0}$ the power of $p_t$ occurring in the prime factorization of $x$. By convention, we set $e_t(x) = 0$ if $p_t$ does not divide $x$. Hence, for every $x \in \N$, 
\begin{equation}\label{EprimeFac}
x = \prod_{i=1}^\infty p_i^{e_i(x)}, 
\end{equation}
where only finitely many terms in the product are not equal to $1$.

\begin{theorem}[Main Result 2]\label{Tmain2}
Let $n \geq 3$ and let $X = (x_1,\dots,x_n) \in \N^n$. Then the following are equivalent for the GCD matrix $S(X)$: 
\begin{enumerate}
\item $S(X)$ is totally nonnegative. 
\item For each $t \in \N$, the sequence $(e_t(x_i))_{i=1}^n$ is monotonic.
\end{enumerate}
\end{theorem}
\noindent Here, by a {\it monotonic sequence}, we mean a sequence that is either non-decreasing or non-increasing. 

The proof of Theorems \ref{Tmain1} and \ref{Tmain2} depend heavily on the notion of Green's matrices, a notion first introduced by Gantmacher and Krein \cite{gantmacher1960oscillation} in their study of oscillatory matrices \footnote{Gantmacher and Krein used the terminology {\it single-pair} matrices instead of {\it Green's matrices}.}.

\begin{definition}[{see Karlin \cite[Chapter 3, \S 3]{karlin1968total}, Gantmacher and Krein \cite[Chapter 2, \S 3]{gantmacher1960oscillation}}]
A $n \times n$ matrix $A = (a_{ij})_{i,j=1}^n$ is said to be a Green's matrix (or a single-pair matrix) if 
\[
a_{ij} = p_{\min(i,j)} q_{\max(i,j)} = \begin{cases}
p_i q_j & \textrm{ if } i \leq j \\
p_j q_i & \textrm{ if } i \geq j.
\end{cases}
\]
for some $p_1,\dots,p_n, q_1,\dots,q_n \in \R$.
\end{definition}

\noindent Our last characterization of totally nonnegative GCD matrices directly involves such matrices. 
\begin{theorem}[Main Result 3]\label{Tmain3}
Let $n \geq 3$ and let $X = (x_1,\dots,x_n) \in \N^n$. Then the following are equivalent for the GCD matrix $S(X)$: 
\begin{enumerate}
\item $S(X)$ is totally nonnegative. 
\item $S(X)$ is a Green's matrix. 
\end{enumerate}
If the $x_i$s are distinct, then $(1)$ and $(2)$ are equivalent to
\begin{enumerate}
\setcounter{enumi}{2}
\item $S(X)^{-1}$ is tridiagonal with nonzero superdiagonal elements.
\end{enumerate}
Moreover, suppose $A := S(X) \in \TN$ and let $\alpha = \{i_1,\dots, i_m\}$ and $\beta = \{j_1,\dots,j_m\}$ be two subsets of $\{1,\dots,n\}$. If $\alpha_s < \beta_{s+1}$ for all $1 \leq s \leq n-1$, then we have
\[
\det A[\alpha,\beta] =\frac{\gcd(x_{k_1},x_n)\gcd(x_1,x_{l_m})}{\gcd(x_1,x_n)^m} \cdot \prod_{i=1}^{m-1} \det \begin{pmatrix}
\gcd(x_{k_{i+1}}, x_n) & \gcd(x_{l_i}, x_n) \\
\gcd(x_1, x_{k_{i+1}}) & \gcd(x_1, x_{l_i})
\end{pmatrix}, 
\]
where $k_s := \min(i_s,j_s)$ and $l_s := \max(i_s,j_s)$. In all other cases, we have $\det A[\alpha,\beta] = 0$. Also, if $A \in \TN$ is non-singular, then 
{\small\[
A^{-1} = \begin{pmatrix}
b_1 & a_2& & & \\
a_2 & b_2 & a_3 & & \\
& \ddots & \ddots & \ddots & \\
 & & a_{n-1} & b_{n-1} & a_n \\
  & & & a_n & b_n
\end{pmatrix}
\]}
where 
\[
a_{i+1} = \frac{\gcd(x_1,x_n)}{\gcd(x_i,x_n)\gcd(x_1,x_{i+1})-\gcd(x_{i+1},x_n)\gcd(x_1,x_i)} \qquad 1 \leq i \leq n-1,
\]
and
\[
b_i = \begin{cases}
- \frac{\gcd(x_2,x_n)}{\gcd(x_1,x_n)} a_2 & (n=1) \\
- \frac{\gcd(x_{i-1},x_n)\gcd(x_1,x_{i+1})-\gcd(x_{i+1},x_n)\gcd(x_1,x_{i-1})}{\gcd(x_1,x_n)}a_i a_{i+1} & (2 \leq i \leq n-1) \\
-\frac{\gcd(x_1,x_{n-1})}{\gcd(x_1,x_n)}a_n & (i=n).
\end{cases}
\]
\end{theorem}

The rest of the paper is structured as follows: we begin by examining the total nonnegativity of $3 \times 3$ GCD matrices in Section \ref{S3}. The proof of our main results are then given in Section \ref{Smain}. Section \ref{Skernels} concludes the paper by examining the total nonnegativity of the kernel $K(x,y) = \gcd(x,y)$, as well as total nonnegativity preservers on GCD matrices. 

\section{The $n=3$ case}\label{S3}
We begin by examining Theorem \ref{Tmain1} in the case where $n=3$.
\begin{proposition}{\label{P3x3}}
Let $a, b, c$ be distinct positive integers with greatest common divisor $d:=\gcd (a, b, c)$, and let $X := (a,b,c)$. Then the following are equivalent: 
\begin{enumerate}
\item $S(X)$ is totally nonnegative
\item  $\gcd (a,c)=d$ and $db \mid ac$.
\item $\gcd(a,b) \gcd(b,c) = b \cdot \gcd(a,c)$.
\end{enumerate}
\end{proposition}

\noindent The following simple lemma will be useful to prove Proposition \ref{P3x3}.

\begin{lemma}\label{Lchain}
Let $x,y,z$ be distinct positive integer. Then 
\begin{equation}\label{Etriple3}
\gcd(x,y) \gcd(y,z) \leq y \cdot \gcd(x,y,z).
\end{equation}
Moreover, equality holds in Equation \eqref{Etriple3} if and only if $\gcd(x,y,z) y \mid xz$.
\end{lemma}
\begin{proof}
Let $d := \gcd(x,y,z)$. Suppose first $d=1$. If $p$ is a prime such that $p \mid \gcd(x,y) \gcd(y,z)$ then $p \mid y$. Also $p \mid \gcd(x,y)$ or $p \mid \gcd(y,z)$ but not both since $\gcd(x,y,z)=1$. It follows that $\gcd(x,y) \gcd(y,z) \leq y$. Equality holds if and only if every divisor of $y$ divides $x$ or $z$, i.e., if and only if $y \mid xz$. This proves the result when $d=1$. The general case follows by replacing $x,y,z$ by $x/d, y/d, z/d$.
\end{proof}

\begin{proof}[Proof of Proposition \ref{P3x3}]
To simplify the notation, let $A := S(X)$: 
\[
A = \begin{pmatrix}
a       & \gcd   (a, b) & \gcd (a, c)\\
\gcd (b, a)       & b & \gcd (b, c) \\
\gcd (c, a)       & \gcd (c, b) & c
\end{pmatrix}.
\]
\smallskip

\noindent$\mathbf{(1) \Rightarrow (2)}.$ Suppose $A \in \TN$. Since the $\{1,2\}, \{2,3\}$ minor of $A$ is nonnegative, we obtain that 
\[
\gcd(a,b)\gcd(b,c) \geq b \cdot \gcd(a,c).
\]
Conversely, by Lemma \ref{Lchain}, $\gcd(a,b)\gcd(b,c) \leq b \cdot \gcd(a,b,c) \leq b \cdot \gcd(a,c)$. It follows that $\gcd(a,c) = \gcd(a,b,c) = d$ and $\gcd(a,b)\gcd(b,c) = b \cdot \gcd(a,b,c)$. By the equality case in Lemma \ref{Lchain}, we conclude that $db \mid ac$. This proves (2).
\smallskip

\noindent$\mathbf{(2) \Rightarrow (3)}.$ Suppose (2) holds. By the equality case in Lemma \ref{Lchain}, we have
\[
\gcd(a,b)\gcd(b,c) = b \cdot\gcd(a,b,c).
\]
 Condition (3) now follows from the assumption that $d = \gcd(a,c)$.
\smallskip

\noindent$\mathbf{(3) \Rightarrow (1)}.$ Suppose now that (3) holds. By Theorem \ref{Tbeslin}, all principal minors of $A$ are nonnegative. Let us examine the remaining minors $A[\alpha,\beta]$ of $A$. By symmetry, there are exactly $3$ cases to consider:
	
\noindent{\bf Case 1.} $\alpha=\{1,2\}, \beta=\{2,3\}$. We have
\[
A[\alpha,\beta] = \begin{pmatrix}
\gcd (a, b) & b\\
\gcd (a, c) & \gcd (b, c)
\end{pmatrix}.
\]
Hence,  $\det A[\alpha,\beta] = \gcd(a,b) \cdot \gcd (b,c) - b \cdot  \gcd(a,c) \geq 0$ by assumption.
	
\noindent{\bf Case 2.} $\alpha=\{1,2\}, \beta = \{1,3\}$. In that case, 
\[
A[\alpha,\beta] = \begin{pmatrix}
a & \gcd (a,c)\\
\gcd (b, a) & \gcd (b, c)
\end{pmatrix}
\]	
By Lemma \ref{Lchain}, we have $ \gcd (b,a)\cdot \gcd (a,c) \leq a \cdot\gcd(a,b,c) \leq a \cdot \gcd(b,c)$. Hence $\det A[\alpha,\beta] \geq 0$.

\noindent{\bf Case 3.} $\alpha = \{2,3\}, \beta = \{1,3\}$. In that case, 
\[
A[\alpha,\beta] = \begin{pmatrix}
\gcd (b, a) & \gcd (b,c)\\
\gcd (c, a) & c
\end{pmatrix}
\]
As in Case 2, by Lemma \ref{Lchain}, we have $\gcd (b,c)\cdot \gcd (c,a) \leq c \cdot\gcd(a,b,c) \leq c \cdot \gcd(b,a)$. Thus $\det A[\alpha,\beta] \geq 0$.

We therefore conclude that $A$ is $\TN$.
\end{proof}

\begin{remark}\label{Rtriple}
Notice that the $(3) \Rightarrow (2)$ implication in Proposition \ref{S3} shows that  
\[
\gcd(a,b)\gcd(b,c) = b \cdot \gcd(a,c) \implies \gcd(a,c) = \gcd(a,b,c).
\]
One can also, of course, proves this directly: suppose $\gcd(a,c) \ne \gcd(a,b,c)$. Then there exists a prime $p$ such that $p \mid a,c$, but $p \nmid b$. Thus $p \mid b\cdot \gcd(a,c)$, but $p \nmid \gcd(a,b)\gcd(b,c)$.
\end{remark}

\section{Proof of the Main Results}\label{Smain}

Before we proceed to the proofs of our main results, we recall some important properties of Green's matrices.

\begin{theorem}[{see Karlin \cite[Chapter 3, Theorem 3.1]{karlin1968total}}]\label{Tgreen}
A Green's matrix $A = (p_{\min(i,j)} q_{\max(i,j)})_{i,j=1}^n$ is TN if and only if all the numbers $p_1,\dots,p_n, q_1,\dots,q_n$ have the same strict sign and 
\[
\frac{p_1}{q_1} \leq \frac{p_2}{q_2} \leq \dots \leq \frac{p_n}{q_n}.
\]
Moreover, for any two subsets $\alpha = \{i_1,\dots,i_m\}$ and $\beta = \{j_1,\dots,j_m\}$ of $\{1,\dots,n\}$, we have
\begin{equation}\label{EGreenDet}
\det A[\alpha,\beta] = p_{k_1} q_{l_m} \cdot \prod_{i=1}^{m-1} \det \begin{pmatrix}
p_{k_{i+1}} & p_{l_i} \\
q_{k_{i+1}} & q_{l_i}
\end{pmatrix}, 
\end{equation}
where $k_s := \min(i_s,j_s)$ and $l_s := \max(i_s,j_s)$ provided $\max(i_s, j_s) < \min(i_{s+1}, j_{s+1})$ for all $s=1,\dots,m-1$. In all other cases, $\det A[\alpha, \beta] = 0$.
\end{theorem}

\begin{theorem}[{Gantmacher and Krein, see \cite[Theorem 2]{barrett1979theorem}}]\label{TGantKrein}
A matrix $A$ is a nonsingular Green's matrix if and only if its inverse is a symmetric tridiagonal matrix with nonzero superdiagonal elements.
\end{theorem}

Using Equation \eqref{EGreenDet}, one can compute the minors of a Green's matrix, and obtain and explicit formula for its inverse via Cramer's rule. 
\begin{theorem}[see e.g.~Yamamoto \cite{yamamoto2001inversion}]\label{TGreenInverse}
Let $A = (p_{\min(i,j)} q_{\max(i,j)})_{i,j=1}^n$ be a non-singular Green's matrix. Then
{\small\[
A^{-1} = \begin{pmatrix}
b_1 & a_2& & & \\
a_2 & b_2 & a_3 & & \\
& \ddots & \ddots & \ddots & \\
 & & a_{n-1} & b_{n-1} & a_n \\
  & & & a_n & b_n
\end{pmatrix}, 
\]}
where 
\[
a_{i+1} = \frac{1}{p_i q_{i+1} - p_{i+1}q_i} \qquad (1 \leq i \leq n-1), 
\]
and 
\[
b_i = \begin{cases}
-\frac{p_2}{p_1(p_1q_2-p_2q_1)} & (n=1) \\
-\frac{p_{i-1}q_{i+1}-p_{i+1}q_{i-1}}{(p_{i-1}q_i -p_iq_{i-1})(p_i q_{i+1}-p_{i+1}q_i)} & (2 \leq i \leq n-1) \\
-\frac{q_{n-1}}{q_n(p_{n-1}q_n-p_nq_{n-1})} & (i=n).
\end{cases}
\]
\end{theorem} 

With the above results in hand, we can now prove our first main result. 

\begin{proof}[Proof of Theorem \ref{Tmain1}]
To simplify the notation, let $A := S(X)$.\hfill

\noindent$\mathbf{(1) \Rightarrow (3).}$ Suppose $A \in \TN_2$ and consider the submatrix $B := A[\{i,j,k\}]$. By assumption $B \in \TN_2$. Moreover, $\det(B) \geq 0$ by Theorem \ref{Tbeslin} and Remark \ref{RBeslin}, and so $B \in \TN$. Properties $(3)$ now follow immediately from Proposition \ref{P3x3}(2).

\noindent$\mathbf{(3) \Leftrightarrow (4)}.$ This follows from the $(2) \Leftrightarrow (3)$ equivalence in Proposition \ref{P3x3}.

\noindent$\mathbf{(4) \Rightarrow (2).}$  Suppose $(4)$ holds and let $1 \leq i \leq j \leq n$. Then 
\begin{align*}
\gcd(x_1,x_j) \gcd(x_i,x_n) &= \frac{\gcd(x_1,x_i)\gcd(x_i,x_j)\gcd(x_i,x_n)}{x_i} \\
&=  \frac{\gcd(x_1,x_i)\gcd(x_i,x_n)}{x_i} \gcd(x_i,x_j) \\
&= \gcd(x_1,x_n) \gcd(x_i,x_j).
\end{align*}
Hence
\begin{equation}\label{EGreenGCD}
a_{ij} = \gcd(x_i,x_j) = \frac{\gcd(x_1,x_j) \gcd(x_i,x_n)}{\gcd(x_1,x_n)} \qquad (1 \leq i \leq j \leq n).
\end{equation}
It follows that $A$ is a Green's matrix with $p_i := \gcd(x_i,x_n)/\gcd(x_1,x_n)$ and $q_j := \gcd(x_1,x_j)$. Hence, by Theorem \ref{Tgreen}, the matrix $A$ is TN if and only if 
\begin{equation}\label{ETNgcd}
\frac{\gcd(x_i,x_n)}{\gcd(x_1,x_i)} \leq \frac{\gcd(x_{i+1},x_n)}{\gcd(x_1,x_{i+1})}
\end{equation}
for all $1 \leq i \leq n-1$. By Remark \ref{Rtriple}, we have $\gcd(x_i,x_k) = \gcd(x_i,x_j,x_k)$. It follows that for all $1 \leq i \leq n$, 
\begin{align}
\gcd(x_i,x_n) &= \gcd(x_i,x_{i+1},\dots, x_n)\label{Egcdp} \\
\gcd(x_1,x_i) &= \gcd(x_1,x_2,\dots,x_i)\label{Egcdq}.
\end{align}
Thus, $\gcd(x_i,x_n) \leq \gcd(x_{i+1},x_n)$ and $\gcd(x_1,x_i) \geq \gcd(x_1,x_{i+1})$ and so Equation \eqref{ETNgcd} always holds. We therefore conclude that $A \in \TN$.

Clearly, we have $\mathbf{(2) \Rightarrow (1)}$ and so the four statements are equivalent.

Finally, suppose $1 \leq i \leq j \leq k \leq l \leq n$. Proceeding as above, we have
\begin{align*}
\gcd(x_i,x_k) \gcd(x_j,x_l) &= \frac{\gcd(x_i,x_j)\gcd(x_j,x_k)\gcd(x_j,x_l)}{x_j} \\
&=  \frac{\gcd(x_i,x_j)\gcd(x_j,x_l)}{x_j} \gcd(x_j,x_k) \\
&= \gcd(x_i,x_l) \gcd(x_j,x_k), 
\end{align*}
as claimed.
\end{proof}

We now prove our second main result, which provides an explicit characterization of the vectors $X = (x_1,\dots,x_n) \in \N^n$ for which the matrix $S(X)$ is TN.

\begin{proof}[Proof of Theorem \ref{Tmain2}] As above, set $A := S(X)$.  By Theorem \ref{Tmain1}(3), $A \in \TN$ if and only if $\gcd(x_i,x_k) = \gcd(x_i,x_j,x_k)$ and $x_j \cdot \gcd(x_i, x_k) \mid x_i x_k$ for all $1 \leq i \leq j \leq k \leq n$. Using the notation in Equation \eqref{EprimeFac}, these two conditions are equivalent to
\begin{align*}
\min(e_t(x_i), e_t(x_k)) &\leq e_t(x_j) \\
e_t(x_j) + \min(e_t(x_i), e_t(x_k)) &\leq e_t(x_i) + e_t(x_k).
\end{align*}
for all $t \in \N$. Hence, $A \in \TN$ if and only if for all $1 \leq i \leq j \leq k \leq n$ and all $t \in \N$, 
\begin{equation}\label{EineqPowers}
\min(e_t(x_i), e_t(x_k)) \leq e_t(x_j) \leq \max(e_t(x_i), e_t(x_k)).
\end{equation}
For $t \in \N$, define a vector 
\[
X_t := (p_t^{e_t(x_1)}, \dots, p_t^{e_t(x_n)}), 
\]
and let $A_t := S(X_t)$. Observe that by Equation \eqref{EineqPowers} and the above arguments, 
\[
A \in \TN \iff A_t \in \TN \textrm{ for all } t=1,\dots,m.
\]
It therefore suffices to resolve the case where all the integers $x_i$ are powers of a given prime $p$. Hence, assume $x_i = p^{\alpha_i}$ for some prime $p \in \N$ and some integers $\alpha_i \in \Z_{\geq 0}$ (not necessarily all distinct). We claim that the matrix $A = S(x_1,\dots, x_n)$ is TN if and only if the sequence $(\alpha_i)_{i=1}^m$ is monotonic. Indeed, suppose $A \in \TN$ and let $1 \leq i \leq j \leq k \leq n$. By Equation \eqref{EineqPowers}, we have 
\begin{equation}\label{Ealphas}
\min(\alpha_i, \alpha_k) \leq \alpha_j \leq \max(\alpha_i,\alpha_k).
\end{equation}
We consider two cases:

\noindent{\bf Case 1.} $\alpha_i < \alpha_j$ If $\alpha_k \leq \alpha_i$, then we obtain from the right hand-side of Equation \eqref{Ealphas} that $\alpha_j \leq \alpha_i$, a contradiction. Thus $\alpha_k > \alpha_i$ and the right hand-side of \eqref{Ealphas} implies that $\alpha_j \leq \alpha_k$.
\smallskip

\noindent{\bf Case 2.} $\alpha_i > \alpha_j$. Similarly, if $\alpha_k \geq \alpha_i$, then the left hand-side of Equation \eqref{Ealphas} implies that $\alpha_j \geq \alpha_i$, a contradiction. Thus $\alpha_k < \alpha_i$ and the right hand-side of \eqref{Ealphas} now implies that $\alpha_j \geq \alpha_k$.

\noindent It follows easily from the above analysis that the sequence $(\alpha_t)_{t=1}^m$ is monotonic.

Conversely, if the sequence $(\alpha_t)_{t=1}^m$ is monotonic, then it satisfies Equation \eqref{Ealphas} and it follows that the $A \in \TN$.
\end{proof}

As showed in the proof of Theorem \ref{Tmain1}, a totally nonnegative GCD matrix is necessarily a Green's matrix (see Equation \eqref{EGreenGCD}). Our last result show that the converse also holds. 

\begin{proof}[Proof of Theorem \ref{Tmain3}]As before, let $A := S(X)$ to simplify the notation.

\noindent$\mathbf{(1) \Rightarrow (2).}$ This was already shown in the proof of Theorem \ref{Tmain1} (see Equation \eqref{EGreenGCD}).

\noindent$\mathbf{(2) \Rightarrow (1).}$ Suppose $A$ is a Green's matrix with entries 
\[
a_{ij} = p_{\min(i,j)} q_{\max(i,j)}.
\] 
Then for $1 \leq i \leq j \leq k \leq n$,
\[
\gcd(x_i, x_j) \gcd(x_j,x_k) = (p_i q_j) (p_j q_k) = (p_j q_j) (p_i q_k) = x_j \cdot \gcd(x_i,x_k).
\]
Thus, by Theorem \ref{Tmain1}(4), the matrix $A$ is TN.

\noindent$\mathbf{(2) \Leftrightarrow (3).}$ This is Gantmacher and Krein's result (see Theorem \ref{TGantKrein}).

The expressions for the minors and the inverse of $A$ are immediate consequences of Theorems \ref{Tgreen} and \ref{TGreenInverse} applied with $p_i, q_i$ as given after Equation \eqref{EGreenGCD}.  
\end{proof}

\section{Totally nonnegative kernels}\label{Skernels}
A natural reformulation of the above results involves the notion of positive semidefinite and totally nonnegative kernels. 
\begin{definition}[see Karlin \cite{karlin1968total}]
Let $X \subseteq \N$. The kernel $K : X \times X \to \R$ is said to be 
\begin{enumerate}
\item {\it positive semidefinite} if the matrix $(K(x_i, x_j))_{i,j=1}^n$ is positive semidefinite for any choice of integers $x_1 < x_2 < \dots x_n$ in $X$ and any $n \geq 1$. 
\item {\it totally nonnegative} if the matrix $(K(x_i, x_j))_{i,j=1}^n$ is totally nonnegative for any choice of integers $x_1 < x_2 < \dots x_n$ in $X$ and any $n \geq 1$. 
\end{enumerate}
\end{definition}
\noindent Positive definite and totally positive kernels are defined analogously. 

\begin{remark}
More generally, for any totally ordered sets $X$ and $Y$, the kernel $K: X \times Y \to \R$ is said to be {\it totally nonnegative} if the matrix $(K(x_i, y_j))_{i,j=1}^n$ is totally nonnegative for any choice of $x_1 < x_2 < \dots, x_n$ in $X$, any choice of $y_1 < y_2 < \dots < y_n$ in $Y$, and any $n \geq 1$. In what follows, we restrict ourselves to the case where $X = Y$.
\end{remark}

Observe that, in the language of kernels, Beslin and Leigh's result (Theorem \ref{Tbeslin}) shows that the kernel
\[
K(x,y) = \gcd(x,y)
\]
is positive definite on $X = \N$ (and hence on any subset of $\N$). Theorem \ref{Tmain2} resolves the analogous problem for total nonnegativity. For simplicity, we only state the result in the case where the cardinality of $X$ is infinite. 
\begin{theorem}\label{Tkernel}
Let $(x_i)_{i=1}^\infty \subseteq \N$ be an increasing sequence of positive integers and let $X = \{x_1,x_2,\dots\}$. Then the following are equivalent: 
\begin{enumerate}
\item  The kernel $K(x,y) = \gcd(x,y)$ is totally nonnegative on $X$.
\item For each $t \in \N$, the sequence $(e_t(x_i))_{i=1}^\infty$ is monotonic.
\end{enumerate}
\end{theorem}

Having characterized the sets $X \subseteq \N$ over which the kernel $K(x,y) = \gcd(x,y)$ is totally nonnegative, it is natural to examine how such kernels can be transformed while remaining totally nonnegative.  More specifically, we are seeking functions $f: \N \to \C$ with the property that $f \circ K$ is totally nonnegative on $X$ whenever $K$ is.  The following result from \cite{belton2017total} shows that for general kernels, not many functions $f$ have this property. 

\begin{theorem}[{\cite[Theorem 2.1]{belton2017total}}]\label{Tfixeddim}
Let $F: [ 0, \infty ) \to \R$ be a function and  let
$\D := \min( m, n )$, where~$m$ and~$n$ are positive
integers. The following are equivalent.
\begin{enumerate}
\item $F$ preserves TN when applied entrywise on $m \times n$ matrices. 
\item $F$ preserves TN when applied  entrywise on $\D \times \D$ matrices. 
\item $F$ is either a non-negative constant or
\begin{enumerate}
\item[(a)] $(\D = 1)$ $F( x ) \geq 0$;
\item[(b)] $(\D = 2)$ $F( x ) = c x^\alpha$ for some $c > 0$ and some
$\alpha \geq 0$;
\item[(c)] $(\D = 3)$ $F( x ) = c x^\alpha$ for some $c > 0$ and some
$\alpha \geq 1$;
\item[(d)] $(\D \geq 4)$ $F( x ) = c x$ for some $c > 0$.
\end{enumerate}
\end{enumerate}
\end{theorem}

As we now show, the situation is very different for the $\gcd$ kernel. Recall that an arithmetic function $f: \N \to \R$ is said to be {\it multiplicative} if $f(xy) = f(x)f(y)$ whenever $\gcd(x,y) = 1$. 
\begin{theorem}\label{TkernelPres}
Let $X \subseteq \N$ and assume the kernel $K(x,y) = \gcd(x,y)$ is totally nonnegative on $X$. Suppose $f: \N \to \C$ satisfies
\begin{enumerate}
\item $f$ is multiplicative, and
\item $f(x) \leq f(y)$  for every $x, y \in \N$ such that $x \mid y$.
\end{enumerate}
Then $f \circ K$ is totally nonnegative on $X$.
\end{theorem}
\begin{proof}
Let $x_1 < x_2 < \dots < x_n$ be a finite subset of $X$. Since $K$ is $\TN$ on $X$, the $\gcd$ matrix $A := (K(x_i, x_j))_{i,j=1}^n$ is $\TN$. By Theorem \ref{Tmain3}, the matrix $A$ is a Green's matrix 
\[
K(x_i, x_j) = p_{\min(i,j)}q_{\max(i,j)}, 
\]
with $p_i := \gcd(x_i,x_n)/\gcd(x_1,x_n)$ and $q_j := \gcd(x_1,x_j)$ --  see Equation \eqref{EGreenGCD}. Using Equations \eqref{Egcdp} and \eqref{Egcdq}, it follows easily that 
\[
\gcd(p_{\min(i,j)},q_{\max(i,j)}) = 1 \qquad \textrm{for all } 1 \leq i \leq j \leq n.
\]
Thus, by Assumption (1) of the theorem, we obtain that 
\[
f(K(x_i, x_j)) = f(p_{\min(i,j)}) f(q_{\max(i,j)}) = \tilde{p}_{\min(i,j)}\tilde{q}_{\max(i,j)},
\]
where $ \tilde{p}_i= f(p_i)$ and $ \tilde{q}_i = f(q_i)$. Since $A$ is $\TN$, we have by Theorem \ref{Tgreen} that 
\[
\frac{p_1}{q_1} \leq \frac{p_2}{q_2} \leq \dots \leq \frac{p_n}{q_n}.
\]
Now, using Equations \eqref{Egcdp} and \eqref{Egcdq}, it follows that for $1 \leq i < n$   
\[
p_i \mid p_{i+1} \quad \textrm{ and } \quad q_{i+1} \mid q_i.
\]
Hence, by Assumption (2) of the theorem, we conclude that $f(p_i) \leq f(p_{i+1})$ and $f(q_i) \geq f(q_{i+1})$, and therefore that 
\[
\frac{\tilde{p}_1}{\tilde{q}_1} \leq \frac{\tilde{p}_2}{\tilde{q}_2} \leq \dots \leq \frac{\tilde{p}_n}{\tilde{q}_n}.
\]
Theorem \ref{Tgreen} now shows that the matrix with entries $f(K(x_i, x_j))$ is $\TN$. Since this is true for any choice of the $x_i$s, we conclude that $f \circ K$ is $\TN$.
\end{proof}

Note that in Theorem \ref{TkernelPres}, we assume that $f(x) \leq f(y)$ only if $x \mid y$. Interestingly, if we assume $f$ is both multiplicative and non-decreasing on $\N$, then $f(n) = n^\alpha$ for some $\alpha \geq 0$. This remarkable result goes back to Erd\H{o}s \cite{erdos1946distribution}. Several authors also found simpler proofs of the result, including Moser and Lambek \cite{moser1953monotone}, Besicovitch \cite{besicovitch1962additive}, Schoenberg \cite{schoenberg1962two}, and Howe \cite{howe1986new}. Erd\H{o}s's original result is stated in terms of additive functions instead of multiplicative ones. 
\begin{theorem}[{see \cite[Theorem XI]{erdos1946distribution}}] 
Let $f: \N \to \R$ satisfy $f(mn) = f(m) + f(n)$ whenever $\gcd(m,n) = 1$, and $f(n+1) \geq f(n)$ for all $n \in \N$. Then $f(n) = C \log n$ for some constant $C$.
\end{theorem}
\noindent In contrast, the weaker hypothesis in Theorem \ref{TkernelPres} is satisfied for much more than power functions. For example, it is satisfied by Euler's totient function $\phi$. 
\begin{corollary}
Let $X \subseteq \N$ and assume the kernel $K(x,y) = \gcd(x,y)$ is totally nonnegative on $X$. Then the kernel $\phi \circ K$ is totally nonnegative. 
\end{corollary}
\begin{proof}
That Euler's totient function is multiplicative is well-known (see e.g.~\cite[Theorem 6.4]{andrews1994number}). To verify the second assumption of Theorem \ref{TkernelPres}, first observe that for any prime number $p$, 
\[
\phi(p^{t+1}) = p^{t+1}-p^t = p^t(p-1) \geq p^{t-1}(p-1) = \phi(p^t).
\]
Using the multiplicativity of $\phi$, it follows easily that $\phi(m) \leq \phi(n)$ if $m \mid n$. 
\end{proof}

\noindent {\bf Acknowledgement.} D.G.~is partially supported by a University of Delaware Research Foundation grant, by a Simons Foundation collaboration grant for
mathematicians, and by a University of Delaware Research Foundation
Strategic Initiative grant. J.W.~is partially supported by the Summer Scholars Program of the University of Delaware.

\bibliographystyle{plain}
\bibliography{biblio}
\end{document}